\documentclass[12pt]{amsart}
\usepackage{graphicx} 
\usepackage{amssymb}
\usepackage{amsthm}
\usepackage{enumitem}
\usepackage{xcolor}
\usepackage{stmaryrd}
\usepackage[margin = 1.15 in]{geometry}
\usepackage{todonotes}
\usepackage{mathrsfs}

\usepackage[
	colorlinks,
	linkcolor={black},
	citecolor={black},
	urlcolor={black}
]{hyperref}

\title[Gap Labelling for Full Shifts]{Gap Labels and Asymptotic Gap Opening \\  for Full Shifts}

\author[D.\ Damanik]{David Damanik}
\address{Department of Mathematics, Rice University, Houston, TX~77005, USA}
\email{damanik@rice.edu}

\author[\'{I}.\ Emilsd\'{o}ttir]{\'{I}ris Emilsd\'{o}ttir}
\address{Department of Mathematics, Rice University, Houston, TX~77005, USA}
\email{irist@rice.edu}

\author[J.\ Fillman]{Jake Fillman}
\address{Department of Mathematics, Texas A\&M University, College Station, TX~77843, USA}
\email{fillman@tamu.edu}

\newcommand{\bbN}{{\mathbb{N}}}
\newcommand{\bbR}{{\mathbb{R}}}
\newcommand{\bbT}{{\mathbb{T}}}
\newcommand{\bbZ}{{\mathbb{Z}}}

\newcommand{\calA}{{\mathcal{A}}}
\newcommand{\calB}{{\mathcal{B}}}

\newcommand{\frS}{{\mathfrak{S}}}
\newcommand{\frA}{{\mathfrak{A}}}

\newcommand{\Hausdorff}{{\mathrm{H}}}
\newcommand{\orb}{{\mathrm{Orb}}}
\newcommand{\range}{{\mathrm{Ran}}}
\newcommand{\QC}{{\mathrm{CR}}}
\newcommand{\clopen}{{\mathrm{clop}}}
\newcommand{\OG}{{\mathrm{OG}}}
\newcommand{\AGO}{{\mathrm{AGO}}}
\newcommand{\diag}{{\mathrm{diag}}}
\newcommand{\jacobi}{{\mathrm{Jacobi}}}
\newcommand{\Schrodinger}{{\mathrm{Schr}}}
\newcommand{\jacobilabels}{{\operatorname{JL}}}
\newcommand{\schrlabels}{{\operatorname{SL}}}
\newcommand{\diaglabels}{{\operatorname{DL}}}
\newcommand{\odsamp}{\mathfrak{p}}
\newcommand{\disamp}{\mathfrak{q}}
\newcommand{\tr}{{\operatorname{Tr}}}

\DeclareMathOperator{\supp}{supp}
\DeclareMathOperator{\Leb}{Leb}

\newcommand{\set}[1]{{\left\{ #1 \right\} } }

\newtheorem{theorem}{Theorem}[section]
\theoremstyle{definition}
\newtheorem{remark}[theorem]{Remark}
\newtheorem{lemma}[theorem]{Lemma}
\newtheorem{coro}[theorem]{Corollary}
\newtheorem{prop}[theorem]{Proposition}
\newtheorem{claim}{Claim}[theorem]
\newtheorem*{claim*}{Claim}
\newtheorem{definition}[theorem]{Definition}

\newenvironment{claimproof}[1][Proof of Claim]{\noindent \underline{#1.} }{\hfill$\diamondsuit$}

\definecolor{maroon}{rgb}{0.35,0,0}

\numberwithin{equation}{section}

\begin{document}

\begin{abstract}
We discuss gap labelling for operators generated by the full shift over a compact subset of the real line. 
The set of Johnson--Schwartzman gap labels is the algebra generated by weights of clopen subsets of the support of the single-site distribution. 
Due to the presence of a dense set of periodic orbits, it is impossible to find a sampling function for which all gaps allowed by the gap labelling theorem open simultaneously.
Nevertheless, for a suitable choice of the single-site distribution, we show that for generic sampling functions, every spectral gap opens in the large-coupling limit. 
Furthermore, we show that for other choices  of weights there are gaps that cannot open for purely diagonal operators.
\end{abstract}

\maketitle

\hypersetup{
	linkcolor={black!30!blue},
	citecolor={red},
	urlcolor={black!30!blue}
}

\section{Introduction}

\subsection{Prologue}
Spurred by recent insights that have enabled the resolution of long-standing open problems related to the spectra of random operators with nontrivial background \cite{AviDamGor2023GAFA} and operators generated by hyperbolic dynamical systems \cite{DF22Doubling}, gap labelling has experienced renewed interest.
This paper is about operator families that are covariant with respect to an ergodic topological dynamical system $(\Omega,T,\mu)$ in which $\Omega$ is a compact metric space, $T$ is a homeomorphism from $\Omega$ to itself, and $\mu$ is a $T$-ergodic Borel probability measure on $\Omega$. 
The topological structure of their spectra can be quite rich: such spectra can be intervals, finite unions of intervals, thin or fat Cantor sets, and so on. 
There is naturally an interest in connecting the topology and dynamics of the base system with the spectra of operators. 
The \emph{gap labelling theorem} provides one such connection.
More specifically, since each operator in question is bounded and self-adjoint, its spectrum is a closed and bounded subset of $\bbR$, so the complement of the spectrum is a union of at most countably many open intervals.
A bounded component of the complement of the spectrum is called a \emph{spectral gap} (or just a \emph{gap}).
Then, in one formulation, the value assumed by the integrated density of states of the operator family at a given energy in a gap lies in the range of the Schwartzman homomorphism corresponding to the suspension of the underlying dynamical system \cite{Johnson1986JDE}.

It is thus completely natural to compute the set of possible labels permitted by the gap labelling theorem, and, having done so, it is of interest to ask whether the set is minimal. 
In other words, does each possible label appear as the label of at least one genuinely open spectral gap for some operator family generated by the underlying dynamics? 
The answer to this question is known to be affirmative anytime the topological dynamical system is \emph{strictly ergodic} (i.e., minimal and uniquely ergodic) and enjoys a non-periodic finite-dimensional factor \cite{AviBocDam2012JEMS}. 
However, for non-uniquely ergodic dynamical systems, it is unclear whether one should expect such a result to remain true. 
Indeed, a non-uniquely ergodic system may have a dense collection of periodic orbits, which makes it impossible for a \emph{single} operator family to have all possible gaps open whenever the set of possible labels is dense in the interval $(0,1)$.
On the other hand, for some mixing systems with connected phase space, such as the doubling map \cite{DF22Doubling}, hyperbolic toral automorphisms \cite{DFGap}, and certain affine automorphisms of suitable connected topological groups \cite{DEF2023JST}, it is known that the label set is not dense and indeed is as small as possible, namely, it is given by $\bbZ$.
Consequently, in those settings, one can open all possible gaps allowed by the gap labelling theorem for the trivial reason that absolutely no gaps whatsoever are allowed.

The full shift over a compact space enjoys a cornucopia of invariant measures, making it a natural starting point for investigating the framework in the non-uniquely ergodic setting.
In view of this discussion, the current manuscript aims to establish the following results in the Johnson--Schwartzman framework for full shifts over compact alphabets: 
\begin{itemize}
\item Compute the set of \emph{possible} (Johnson--Schwartzman) labels corresponding to shift-ergodic measures whose topological support is the full shift.
In view of Johnson's theorem \cite{Johnson1986JDE}, this is precisely given by intersecting the \emph{Schwartzman group} of the underlying dynamical system, which we denote by $\frS(\Omega,T,\mu)$, with the interval $(0,1)$. 
For a precise definition, see Subsection~\ref{ss:johnsonGL}.\medskip
\item Demonstrate that, for some choices of a full shift and a fully supported shift-ergodic measure, all possible labels appear as labels of genuinely open spectral gaps.\medskip
\item Moreover, in those same cases, for generic sampling functions, all spectral gaps open in the large-coupling limit.\medskip
\item However, for other choices of the (fully supported) ergodic measure, there are possible gap labels that nevertheless cannot appear as the label of an open gap for \emph{any diagonal} operator family over the given base dynamics.\bigskip
\end{itemize}

\subsection{Results}
To set notation and state results, suppose $\mu_0$ is a Borel probability measure with compact topological support $\supp\mu_0 =: \mathcal{A} \subseteq \bbR$. Put 
\[\Omega = \mathcal{A}^\bbZ
:= \{(\omega_n)_{n \in \bbZ} : \omega_n \in \calA \text{ for every } n \in \bbZ\},
\] let  $T:\Omega \to \Omega$ denote the shift \[[T\omega]_n = \omega_{n+1},\] 
and let $\mu = \mu_0^\bbZ$ denote the product measure.
We will refer to $(\Omega,T,\mu)$ as the \emph{full shift} generated by $(\calA,\mu_0)$ (or the \emph{full shift over} $\calA$ when $\mu_0$ is clear from context).

Given a pair of continuous functions $\disamp:\Omega \to \bbR$ and $\odsamp : \Omega \to \bbR_{\geq 0}$, the associated dynamically defined Jacobi matrices $J_\omega =  J_{\disamp,\odsamp,\omega}$ act on $\ell^2(\bbZ)$ by
\begin{equation}\label{eq:jacobidef}
[J_\omega \psi](n) = \odsamp(T^{n-1}\omega)\psi(n-1) + \disamp(T^n\omega) \psi(n)+ \odsamp(T^n\omega)\psi(n+1).
\end{equation}
We will denote the $\mu$-almost sure spectrum of the family $\{J_{\disamp, \odsamp,\omega}\}_{\omega \in \Omega}$ by $\Sigma_{\disamp,\odsamp} = \Sigma_{\disamp,\odsamp}(\Omega,T,\mu)$ and the integrated density of states by $k_{\disamp, \odsamp}$.
The case $\odsamp \equiv 1$ leads to an ergodic family of \emph{Schr\"odinger operators}, which we denote by $H_{\disamp,\omega} := J_{\disamp,1,\omega}$. 
We direct the reader to Section~\ref{sec:background} for a more thorough discussion.

If one chooses the sampling functions to be $\disamp(\omega) = \omega_0$ and $\odsamp \equiv 1$, the corresponding ergodic family is the \emph{Anderson model} with single-site distribution $\mu_0$, and its spectrum is (almost surely) given by the expression \cite{KunzSouillard1980}
\begin{equation}
\Sigma = \calA + [-2,2]
= \{x+ y : x \in \calA \text{ and } -2\le y \le 2\}.
\end{equation}
 From this, one knows that the spectrum consists of compact intervals of length at least four, separated by finitely many open gaps, and one can (given $\calA$) compute these quantities.
 This explicit expression for the spectrum is particular to this choice of sampling function. 
However, $(\Omega,T,\mu)$ is a topological dynamical system to which the gap labelling theorem applies, so it is natural to compute its Schwartzman group $\frS (\Omega,T,\mu)$ and hence the set of allowable labels to see what other spectral gaps may arise for different choices of $\disamp$ and $\odsamp$.

\begin{definition}
The following definitions establish the framework needed to formulate the gap labelling theorem for the full shift.
\begin{enumerate}[label = {\rm (\alph*)}]
\item If $W \subseteq \bbR$, then $\bbZ[W]$ denotes the \emph{algebra} generated by $W$, that is, the smallest algebra containing $W$. \item If $Y$ is a topological space, one says that $C \subseteq Y$ is \emph{clopen} if $C$ is both closed and open. We denote by $\clopen(Y)$ the collection of clopen subsets of $Y$.
\item If $\nu$ is a measure on a set $S$ and $\mathcal{P}$ is a collection of measurable subsets of $S$, then
\[ \nu(\mathcal{P})= \set{\nu(P) : P \in \mathcal{P}}.\]
Likewise, if $\mathcal{F}$ is a collection of integrable functions, we write
\[ \nu(\mathcal{F})= \set{ \int 
\! f \, d\nu : f \in \mathcal{F}}.\]
\end{enumerate}
\end{definition}

With these preliminaries, we can formulate the conclusion of the gap labelling theorem for the case under consideration: the Schwartzman group is precisely the algebra generated by weights of clopen subsets of the support with respect to the single-site distribution:
\begin{theorem}\label{t:rand:GL:main}
Let $\calA \subseteq \bbR$ denote a compact set, $\mu_0$ a Borel probability measure with $\supp\mu_0 = \calA$, and consider the product space $(\Omega,T,\mu)$ where $\Omega=\calA^\bbZ$, $\mu=\mu_0^\bbZ$, and $T$ denotes the shift $[T\omega]_n = \omega_{n+1}$. Then the Schwartzman group of $(\Omega,T,\mu)$ is given by
\begin{equation}
    \frS(\Omega,T,\mu) = \bbZ[\mu_0(\clopen(\calA))].
\end{equation}
\end{theorem}

\begin{remark}
Let us emphasize that $\calA$ is viewed as a compact topological space in the relative topology that it inherits as a subset of $\bbR$, and thus $\clopen(\calA)$ should be understood as the set of subsets of $\calA$ that are both relatively closed and relatively open.

We also mention that, in the cases when both can be computed, this agrees with the labels that are computed via K-theory; see \cite[Theorem~4.6]{BBG1992RMP}.
\end{remark}

\begin{coro} \label{coro:AconnfrS=Z}
In the setting of Theorem~\ref{t:rand:GL:main}, $\frS(\Omega,T,\mu) = \bbZ$ if $\calA$ is connected. 
\end{coro}
\begin{remark}
By the gap labelling theorem, the integrated density of states associated with a given ergodic family assumes values in the set $\frS \cap (0,1)$ in spectral gaps \cite{DFGap, DFZ, Johnson1986JDE}.
 In particular, if $\frS = \bbZ$, then $\frS \cap (0,1) = \emptyset$ and there can be no spectral gaps, so the almost-sure spectrum of any ergodic family with base dynamics $(\Omega,T,\mu)$ is an interval whenever Corollary~\ref{coro:AconnfrS=Z} is applicable.
\end{remark}

Once the set of possible labels has been identified, one would like to know whether it is minimal. 
Put differently, having identified a set that is guaranteed to be a superset of all possible labels, it is then of interest to determine whether every element of the putative label set occurs as the label of a genuine open gap. 
To that end, given a possible label 
\begin{equation}
\ell \in \frS_0 = \frS_0(\Omega,T,\mu) := \frS(\Omega,T,\mu) \cap (0,1),
\end{equation}
we say that $\Sigma_{\disamp, \odsamp}$ \emph{has an open gap with label} $\ell$ if there is a bounded component $I$ of $ \bbR \setminus \Sigma_{\disamp,\odsamp}$ such that $k_{\disamp,\odsamp}(E) = \ell$ for $E \in I$, where $k_{\disamp, \odsamp}$ denotes the integrated density of states (see Section~\ref{sec:background}).
We then denote
\begin{align} 
\nonumber
\OG_\jacobi(\ell) 
& = \OG_\jacobi(\Omega,T,\mu,\ell) \\
\label{eq:OGdef}
&= \{(\disamp, \odsamp) \in C(\Omega,\bbR) \times C(\Omega,\bbR_{\geq 0}) :  \substack{\Sigma_{\disamp, \odsamp}\text{ has an open gap} \\ \text{with label } \ell} \}.
\end{align}
For the Schr\"odinger and diagonal operator families, we also may consider
\begin{align}\label{eq:OGschrdef}
\OG_\Schrodinger(\ell)
= \OG_\Schrodinger(\Omega,T,\mu,\ell) 
 & =\{\disamp \in C(\Omega,\bbR) : (\disamp,1) \in \OG_\jacobi(\ell)\}, \\
\label{eq:OGdiagdef}
\OG_\diag(\ell) 
= \OG_\diag(\Omega,T,\mu,\ell) 
& = \{ \disamp \in C(\Omega,\bbR) : (\disamp,0) \in \OG_\jacobi(\ell) \}.
\end{align}
Let us note that various quantities such as $\Sigma_{\disamp, \odsamp}$ and $\OG_\bullet(\ell)$ depend on the base dynamical system $(\Omega,T,\mu)$, but we leave it out of the notation when it is clear from context.

Now, one can ask whether the label set $\frS_0$ is minimal in two senses:
\begin{itemize}
\item Weak sense: Is $\OG_\jacobi(\ell)$ nonempty for each $\ell \in \frS_0$? If so, how large is it (in, say, the topological sense)? What about $\OG_\Schrodinger(\ell)$ and $\OG_\diag(\ell)$?

\item Strong sense: Is $\AGO_\jacobi = \AGO_\jacobi(\Omega,T,\mu) := \bigcap_{\ell \in \frS_0} \OG_\jacobi(\Omega,T,\mu,\ell)$ nonempty? If so, how large is it (in, say, the topological sense)? What about the analogous sets
\begin{align*}
 \AGO_\Schrodinger = \AGO_\Schrodinger(\Omega,T,\mu) 
 & := \bigcap_{\ell \in \frS_0} \OG_\Schrodinger(\Omega,T,\mu,\ell), \\ 
 \AGO_\diag
 = \AGO_\diag(\Omega,T,\mu)
 & := \bigcap_{\ell \in \frS_0} \OG_\diag(\Omega,T,\mu,\ell)?
 \end{align*}
\end{itemize}

We explicitly single out the strong version of the question since this is often what is pursued in the context of strictly ergodic base dynamics:
 
\begin{itemize}
\item Consider $\Omega = \bbT:= \bbR/\bbZ$, $T:\omega \mapsto \omega + \alpha$ (with $\alpha$ irrational), which has a unique invariant Borel probability measure $\mu$ given by Lebesgue measure on $\bbT$.
For this ergodic system, the set of possible labels is $\frS_0 = (\bbZ + \alpha\bbZ ) \cap (0,1)$. The Dry Ten Martini problem is then equivalent to asking whether $\disamp_\lambda \in \AGO_\Schrodinger(\bbT, \cdot+\alpha, \Leb)$ for every $\lambda \neq 0$, where 
\[\disamp_\lambda(\omega)  = 2\lambda \cos(2\pi\omega).\]
This has been studied by many people over the years \cite{AvilaJitomirskayaJEMS2010, AvilaYouZhou2024, BorgniaSlager, ChoiEllYui, LiuYuan2015, Puig2004, Riedel2012}.

\item Another recent success shows a similar result for arbitrary \emph{Sturmian} subshifts. Given $\alpha \in \bbT$ irrational and $\rho \in \bbT$
the corresponding lower and upper mechanical words are given by
\[ s_{\alpha,\rho}(n) = \chi_{[1-\alpha)} (n \alpha + \rho ),
\quad s_{\alpha,\rho}'(n) = \chi_{(1-\alpha]}  ( n\alpha + \rho ),  \]
and the Sturmian subshift with slope $\alpha$ is given by $\Omega_\alpha = \{s_{\alpha,\rho} : \rho \in \bbT\} \cup  \{s_{\alpha,\rho}' : \rho \in \bbT\} $.
Equipping $\Omega_\alpha$ with the shift transformation, $T$, it is well-known that $(\Omega_\alpha,T$) is strictly ergodic \cite{Queffelec1987}; denote by $\mu$ the unique invariant Borel probability measure. In the present terminology, Band, Beckus, and Loewy show that $\disamp_\lambda$ given by $\disamp_\lambda(\omega) = \lambda \omega_0$ lies in $\AGO_\Schrodinger(\Omega_\alpha,T,\mu)$ for every $\lambda>0$ \cite{BanBecLoe2024drytenmartinisturm}, which extends previous work of Raymond \cite{Raymond1997} (see also \cite{BBBRT2024raymond}).

\item Once we consider a general strictly ergodic system, it is reasonable to ask about typical sampling functions rather than explicit sampling functions.
In that context, Avila, Bochi, and Damanik showed that for any strictly ergodic dynamical system having a non-periodic finite-dimensional factor, $\AGO_\Schrodinger$ contains a dense $G_\delta$ subset of $C(\Omega,\bbR)$ \cite{AviBocDam2012JEMS}, with a similar result due to Damanik and Li for $\AGO_\jacobi$ \cite{DamLon2024gapOpening}. 
\end{itemize}

In the setting of a full shift over a non-connected alphabet, it is quite natural to ask these questions.
However, one can quickly observe using the presence of a dense set of periodic points that there is \emph{no} sampling function $\disamp$ such that the associated family of ergodic Schr\"odinger operators has all gaps open simultaneously. 
We thus have the following observation in the current setting.

\begin{theorem} \label{t:DTMempty}
For the full shift over a nonconnected alphabet, $\AGO_\Schrodinger = \emptyset$. 
\end{theorem}

However, for certain choices of the measure on the underlying system, one can open all gaps for diagonal operators and all gaps for Schr\"odinger operators in the large-coupling limit:

\begin{theorem} \label{thm:gapsOpenLargeCoupling}
Let $\calA = \{0,1,2,\ldots, m-1\}$, let  $\mu_0$ be given by $\mu_0(\{j\})=1/m$ for each $j \in \calA$, and let $(\Omega_m,T,\mu)$ be the full shift generated by $(\calA,\mu_0)$. Then:
\begin{enumerate}[label = {\rm (\alph*)}]
\item For generic $\disamp \in C(\Omega_m,\bbR)$, all gaps open in the large coupling limit. 
That is, there is a dense $G_\delta$ set $\mathcal{G} \subseteq C(\Omega,\bbR)$ such that for each $\disamp \in \mathcal{G}$ and every $\ell \in \frS_0$, there exists $\lambda_0(\disamp,\ell)$ such that $\lambda  \disamp \in \OG_\Schrodinger(\ell)$ for all $\lambda \geq \lambda_0(\disamp,\ell)$.
\item   $\OG_\Schrodinger(\ell) \neq \emptyset$ for all $\ell \in \frS_0$.
\item  $\AGO_\diag \neq \emptyset$ and $\AGO_\jacobi \neq \emptyset$.
\end{enumerate}

\end{theorem}

Once one understands that in some cases all possible labels can appear whereas in other cases some labels are forbidden, it is natural to ask what are the sets of labels that appear as labels of genuine open gaps.
We emphasize that in the setting of strictly ergodic base dynamics with a non-periodic finite-dimensional factor, the answer to this question is trivially $\frS_0$ by \cite{AviBocDam2012JEMS}: generic sampling functions have \emph{all} possible gaps open.
Thus, the non-uniquely ergodic setting offers a chance to observe genuinely new behavior.

Concretely, let us consider
\begin{align*}
\jacobilabels
= \jacobilabels(\Omega,T,\mu) 
& = \set{ \ell \in \frS_0 : \OG_\jacobi(\Omega,T,\mu,\ell) \neq \emptyset}, \\
\schrlabels
= \schrlabels(\Omega,T,\mu) 
& = \set{ \ell \in \frS_0 : \OG_\Schrodinger(\Omega,T,\mu,\ell) \neq \emptyset}, \\
\diaglabels 
= \diaglabels(\Omega,T,\mu)
& = \set{ \ell \in \frS_0 : \OG_\diag(\Omega,T,\mu,\ell) \neq \emptyset}.
\end{align*} 

\begin{theorem} \label{thm:labelSetIncl}
For any $(\Omega,T,\mu)$:
\begin{equation}
\diaglabels
\subseteq \schrlabels
\subseteq \jacobilabels
\subseteq \frS_0.
\end{equation}
\end{theorem}

We can show that in some cases, at least one of these inclusions is \emph{strict}:

\begin{theorem}\label{thm:gapsClosed}
Let $\calA = \{0,1\}$, and let  $\mu_0$ be given by $\mu_0(\{0\})=\beta$ and $\mu_0(\{1\})=1-\beta$ with $\beta$ transcendental. For the full shift generated by $(\calA,\mu_0)$ there exists $\ell \in \frS_0$ such that $\OG_\diag(\ell) = \emptyset$, that is, $\diaglabels \subsetneq \frS_0$.
\end{theorem}

\begin{remark} Theorem~\ref{thm:gapsClosed} and its proof suggest some interesting questions.
\begin{enumerate}[label = {\rm (\alph*)}]
\item The approach we employ shows that for the examples we produce, there is some label allowed by the gap labelling theorem for which there cannot be any \emph{diagonal} ergodic family realizing that label as the value of the integrated density of states on a genuinely open spectral gap.
In view of Theorem~\ref{thm:labelSetIncl}, if a gap with label $\ell$ can be opened with a diagonal operator, then it can also be opened for Schr\"odinger and Jacobi operators.
However, the converse does not necessarily hold, so the conclusion of Theorem~\ref{thm:gapsClosed} does not preclude the existence of an ergodic family of \emph{Schr\"odinger}  operators having a spectral gap with the label produced by the theorem.
Is it true that $\schrlabels \subsetneq \frS_0$ in the setting of Theorem~\ref{thm:gapsClosed}?
More broadly, can one produce an example of a dynamical system such that there is a possible gap label for which there is no Schr\"odinger family exhibiting an open spectral gap with that label?
\medskip

\item 
The approach to gap labelling via the Schwartzman group (as well as the approach via $K$-theory) realizes the set of possible gap labels by intersecting a group with the interval $(0,1)$, so any integer linear combination of possible labels that happens to belong to $(0,1)$ is again a possible label.
However, it is not obvious to the authors that set of values that could in principle be assumed by the integrated density of states of an ergodic family of Schr\"odinger operators in spectral gaps needs to possess such an arithmetic structure.
Indeed, this is exactly the mechanism that is employed in the proof of the theorem to produce a label that cannot be realized by any family of diagonal operators.
\medskip

\item We also remark that questions about $\diaglabels$ in particular are most interesting in the setting of dynamics on a totally disconnected space.
For instance, if $\Omega$ is connected and $\mu$ is a fully-supported $T$-ergodic Borel probability measure on $\Omega$, it is not hard to see that $\diaglabels(\Omega,T,\mu) = \emptyset$.
\end{enumerate}

\end{remark}

The outline of the paper is as follows.
In Section~\ref{sec:background}, we recall background information about ergodic operators and the gap labelling theorem.
Section~\ref{subsec:fullshift} computes the set of labels associated with the full shift, in particular proving Theorem~\ref{t:rand:GL:main}.
In Section~\ref{sec:opengaps}, we discuss results in which gaps can be opened, and we prove Theorem~\ref{thm:gapsOpenLargeCoupling}.
Section~\ref{sec:nonopengaps} discusses obstructions to gap opening and contains the proofs of Theorems~\ref{t:DTMempty}, \ref{thm:labelSetIncl}, and \ref{thm:gapsClosed}.

\subsection*{Acknowledgements}
D.D.\ and \'{I}.E.\ were supported in part by National Science Foundation grants DMS-2054752 and DMS-2349919.
J.F.\ was supported by National Science Foundation grants DMS-2213196 and DMS-2513006. 
The authors gratefully acknowledge support from the American Institute of Mathematics through a recent SQuaRE program and thank Anton Gorodetski for helpful conversations.

\section{Background} \label{sec:background}
\subsection{Ergodic Operators}
Let us briefly recall the main objects. 
We direct the reader to \cite{ESO1, ESO2} for proofs and further discussion. 
We define a \emph{topological dynamical system} to consist of a pair $(\Omega, T)$ in which $\Omega$ is a compact metric space and $T:\Omega \to \Omega$ is a homeomorphism; in particular, we consider invertible dynamics. 
A Borel probability measure on $\Omega$ is called $T$-\emph{invariant} if $\mu(T^{-1}B) = \mu(B)$ for all Borel sets $B$ and $T$-\emph{ergodic} if it is $T$-invariant and any $T$-invariant measurable function is $\mu$-a.e.\ constant.
We adopt the  convention
\begin{equation}
\supp\mu = \Omega,
\end{equation}
which is standard in the current setting and is also non-restrictive, since the restriction of $\mu$ to $\supp\mu$ is then a fully supported ergodic measure on the topological dynamical system $(\supp\mu, T|_{\supp\mu})$.

Given a topological dynamical system $(\Omega,T)$ with an ergodic measure $\mu$, an ergodic family of \emph{Jacobi matrices} is defined by a choice of $\disamp \in C(\Omega,\bbR)$ and $\odsamp \in C(\Omega,\bbR_{\geq 0})$ and given by \eqref{eq:jacobidef}, that is,
\begin{equation} \label{eq:ejmdef}
[J_\omega \psi](n) =
\odsamp(T^{n-1}\omega)\psi(n-1) + \disamp(T^n\omega) \psi(n)+ \odsamp(T^{n}\omega)\psi(n+1), \quad \psi \in \ell^2(\bbZ).
\end{equation}
We write $H_\omega = J_{\disamp,1\omega}$ and $V_\omega = J_{\disamp,0,\omega}$ for corresponding ergodic Schr\"odinger operators and ergodic diagonal operators, respectively.
We also write $V_\omega$ for the potential $V_\omega(n) = \disamp(T^n\omega)$.

With setup as above, there is a set $\Sigma = \Sigma_{\disamp, \odsamp}$ such that
\begin{equation}
\sigma(J_{\disamp, \odsamp,\omega}) = \Sigma, \quad \mu\text{-a.e.\ } \omega \in \Omega.
\end{equation}
We call $\Sigma$ the \emph{almost-sure spectrum} of the family $\{J_\omega\}_{\omega \in \Omega}$.
The \emph{density of states measure}, $\kappa=\kappa_{\disamp, \odsamp}$, is defined by 
\begin{equation}
\int \! h  \, d\kappa
= \int_\Omega  \! \langle \delta_0, h(J_\omega) \delta_0 \rangle \, d\mu (\omega), 
\quad h \in C(\bbR).
\end{equation}
The \emph{integrated density of states}, $k = k_{\disamp, \odsamp}$ (which depends on the base dynamics $(\Omega,T,\mu)$ as well as the sampling functions  $\disamp$ and $\odsamp$), is the accumulation function of $\kappa$:
\begin{equation}
k(E) = \int  \!  \chi_{_{(-\infty,E]}} \, d\kappa.
\end{equation}

\subsection{Johnson--Schwartzman Gap Labelling} \label{ss:johnsonGL}
Let us describe the Johnson--Schwartzman approach to gap labelling; see \cite{ESO1} or \cite{DFGap} for proofs and further discussion.
Given an ergodic topological dynamical system $(\Omega,T,\mu)$, its \emph{suspension} is denoted by $(X,\tau,\nu)$. The space $X$ is given by
\begin{equation}
X  
= \Omega \times \bbR /\!\!\sim, \text{ where } (T^n\omega,t) \sim (\omega,t+n) \text{ for } \omega \in \Omega, \ t \in \bbR, \ n \in \bbZ,
\end{equation}
or sometimes (equivalently) $X$ is represented as $\Omega \times [0,1]/\!\!\sim$, where the relation is given by $(\omega,1)\sim(T\omega,0)$.
The $\bbR$-action (on the first realization of $X$) is defined by $\tau^s([\omega,t])=[\omega,s+t]$ and the suspended measure $\nu$ is given by
\begin{equation}
\int_{X} \! f \, d\nu
= \int_\Omega \int_0^1\!  f([\omega,t]) \, dt \, d\mu(\omega).
\end{equation}
Given a continuous function $\phi:X \to \bbT = \bbR/\bbZ$ and $x \in X$, one can lift the function $\phi_x: t \mapsto \phi(\tau^tx)$ to a continuous map $\widetilde{\phi}_x:\bbR \to \bbR$ satisfying $ \widetilde\phi_x \operatorname{mod}\bbZ = \phi_x $. The limit
\begin{equation}  \lim_{t\to\infty} \frac{\widetilde{\phi}_x(t)}{t} \end{equation}
exists for $\nu$-a.e.\ $x \in X$, is $\nu$-a.e.\ constant, and only depends on the homotopy class of $\phi$. Denoting by $C^\sharp(X,\bbT)$ the set of homotopy classes of maps $X \to \bbT$, the induced map $\mathfrak{A}_{\nu}:C^\sharp (X,\bbT) \to \bbR$ is called the \emph{Schwartzman homomorphism} (named after \cite{Schwarzmann1957Annals}) and its range is a countable subgroup of $\bbR$, known as the \emph{Schwartzman group}, denoted
\begin{equation} \label{eq:rand:GL:schwartzmanDef}
\frS(\Omega,T,\mu) = \mathfrak{A}_{\nu}(C^\sharp(X,\bbT)).
\end{equation}

Johnson's gap labelling theorem \cite{DFZ, Johnson1986JDE} asserts that the integrated density of states, $k$, assumes values only in the Schwartzman group for energies in the spectral gaps of the ergodic family, that is,
\begin{equation}
k_{\disamp,\odsamp}(E) \in \frS(\Omega,T,\mu), \ \forall E \in \bbR \setminus \Sigma_{\disamp,\odsamp}.
\end{equation}
Here we observe that the quantity $k_{\disamp,\odsamp}(E)$ depends on $\disamp$ and $\odsamp$, but the set $\frS(\Omega,T,\mu)$ does not.

\section{Gap Labelling for Full Shifts}\label{subsec:fullshift}
Given that the Schwartzman group is characterized as the range of the Schwartzman homomorphism on homotopy classes of maps, the first step is to find a suitable set of representatives of $C^\sharp(X,\bbT)$.  Given $g \in C(\Omega,\bbZ)$, one defines $\bar{g} \in C(X,\bbT)$ by
\begin{equation} \label{eq:rand:GL:bargDEF}
 \bar{g}([\omega,t]) = t\cdot g(\omega) \ \mathrm{mod} \ \bbZ, \quad (\omega,t) \in \Omega \times [0,1].
 \end{equation}
It turns out that this captures all continuous functions $X \to \bbT$ modulo homotopy, which in turn allows one to fully describe the Schwartzman group as the set of integrals of continuous integer-valued functions against the ergodic probability measure.

\begin{theorem} \label{thm:rand:GL:main}
Assume $\calA \subseteq \bbR$ is compact, $\mu_0$ is a probability measure with $\supp\mu_0 = \calA$, $\Omega = \calA^\bbZ$, $\mu=\mu_0^\bbZ$, and $T$ denotes the shift on $\Omega$. Let $(X,\tau, \nu)$ denote the suspension of $(\Omega,T,\mu)$.

\begin{enumerate}[label = {\rm (\alph*)}]
\item Every $h \in C(X,\bbT)$ is homotopic to $\bar{g}$ for some $g \in C(\Omega,\bbZ)$.

\item For each $g \in C(\Omega,\bbZ)$,
\begin{equation}
    \frA_{\nu}(\bar{g}) = \int_\Omega \! g \,d\mu.
\end{equation}
\item Consequently,
\begin{equation}
    \frS(\Omega,T,\mu) = \mu(C(\Omega,\bbZ)).
\end{equation}
\end{enumerate}
\end{theorem}

\begin{proof}
(a) Let $h \in C(X,\bbT)$ be given. Since $\calA \subseteq \bbR$, the map $\varphi:\Omega \to \bbT$ sending $\omega$ to $h([\omega,0])$ is nullhomotopic. In particular by \cite[Proposition~3.9.9]{ESO1}, there exists $\widetilde\varphi \in C(\Omega,\bbR)$ such that 
\begin{equation}
\pi(\widetilde\varphi(\omega)) = h([\omega,0])\text{ for all } \omega \in \Omega,
\end{equation}
 where  $\pi:\bbR \to \bbT$ denotes the canonical quotient map.

For each $\omega \in \Omega$, the map $t \mapsto h([\omega,t])$ gives a continuous function $[0,1] \to \bbT$, so it lifts to a unique continuous function $t \mapsto \widetilde\varphi_t(\omega) \in \bbR$ such that 
\begin{equation}
\pi(\widetilde\varphi_t(\omega)) = h([\omega,t]), \text{ and }\widetilde\varphi_0(\omega) = \widetilde\varphi(\omega).
\end{equation}
 By uniform continuity of $h$, one can check that $(\omega,t)\mapsto \widetilde\varphi_t(\omega)$ is (uniformly) continuous on $\Omega \times [0,1]$. Furthermore, since 
$$
\pi(\widetilde\varphi_1(\omega))
= h([\omega,1]) 
= h([T\omega,0]) 
= \pi(\widetilde\varphi_0(T\omega)),
$$
we see that $g(\omega) := \widetilde\varphi_1(\omega) - \widetilde\varphi_0(T\omega)$ is a continuous integer-valued function. Since $h$ is homotopic to $\bar{g}$, we are done with (a). 

(b) Given $g \in C(\Omega,\bbZ)$ and $x \in X$, consider $\phi_x(t) = \bar{g}(\tau^t x)$ and the corresponding lift $\widetilde{\phi}_x$. From the definition of $\bar{g}$ and Birkhoff's ergodic theorem, one has 
\begin{equation}
 \lim_{n\to\infty} \frac{\widetilde{\phi}_{[\omega,0]}(n)}{n} 
= \lim_{n\to\infty} \frac{1}{n}\sum_{j=0}^{n-1} g(T^j\omega) 
= \int \! g\, d\mu
\end{equation}
for $\mu$-a.e.\ $\omega$. Since $\phi_{[\omega,s]}(t)  = \phi_{[\omega,0]}(s+t)$, the previous statement yields the desired conclusion.

(c) This follows from (a) and (b).
\end{proof}

\begin{definition}
We call $\Xi \subseteq \Omega$ a \emph{clopen rectangle} if there are $n \in \bbZ$, $k \in \bbN$, and clopen sets $B_j \subseteq \calA$ for $n \le j < n + k$ such that
\[ \Xi = \{\omega \in \Omega : \omega_j \in B_j \ \forall n \le j < n+k \}.\]
Defining $B_j = \calA$ for $j<n$ and $j \geq n+k$, we can also write this as
\[\Xi = \prod_{j \in \bbZ} B_j.\]
Let $\QC(\Omega)$ denote the set of all clopen rectangles in $\Omega$.
\end{definition}

\begin{prop} \label{prop:cloptoQC}
    Every nonempty $C \in \clopen(\Omega)$ can be written as a disjoint union of finitely many clopen rectangles.
\end{prop}

Proposition~\ref{prop:cloptoQC} will follow from a finite-dimensional version and some basic facts about the product topology. We expect this is well-known, but we were unable to find a reference in this formulation, so we give a proof to keep the paper more self-contained.

\begin{lemma} \label{lem:rand:GL:clopProdcloprect}
    Suppose $\calA_1$ and $\calA_2$ are compact metric spaces, and denote $\calB = \calA_1 \times \calA_2$. Then, every nonempty $C \in \clopen( \calB)$ can be written as a disjoint union of finitely many clopen rectangles. 
\end{lemma}

\begin{proof}
    Let  $C  \in \clopen(\calB)$ be given, and let $\pi_{j}:\calB \to \calA_j$ denote the coordinate projections. Note that both $\pi_1$ and $\pi_2$ map closed sets to closed sets, open sets to open sets, and hence clopen sets to clopen sets.

    Let us introduce some notation. Given $a = (a_1,a_2) \in \calB$, define the associated sections of $C$ by
    \begin{align*}
        C_1^{a} & = \pi_{1}(C \cap \pi_2^{-1}(\{a_2\})) = \set{s \in \calA_1 : (s,a_2) \in C},\\
        C_2^{a} & = \pi_{2}(C \cap \pi_1^{-1}(\{a_1\})) = \set{t \in \calA_2 : (a_1,t) \in C}.
    \end{align*}

    \begin{claim}
         For each  $a \in \calB$, $C_1^{a}$ and $C_2^a$ are clopen.
    \end{claim}
    \begin{claimproof}
    Without loss of generality, consider $j=1$.
            Since $C$ is closed, $\pi_2$ is continuous, and $\pi_1$ maps closed sets to closed sets, $C_1^{a}$ is closed. 
            
            To see that $C_1^{a}$ is open, assume  $s \in C_1^{a}$. Then $(s, a_2) $ belongs to $C$, so, since $C$ is open, there is a basic neighborhood $U \times V$ of $(s,a_2)$ contained in $C$. This, in turn, implies $U \subseteq C_1^{a}$, whence $C_1^a$ is open.
    \end{claimproof}
\bigskip

    Now, let $x= (x_1,x_2) \in C$ be given. We will show that $x$ can be enclosed in a clopen rectangle contained in $C$. As a first step, we show that $x$ can be separated from any element of the complement of $C$ by a clopen rectangle.

\begin{claim} \label{claim:randgaplabel:clopSepar}
For any $y = (y_1, y_2) \in  \calB \setminus C$, there is a clopen rectangle $R=R(y)$ containing $x$ that does not contain $y$.
\end{claim}
\begin{claimproof}
Such a clopen rectangle is supplied by: 
\[R(y)=
\begin{cases}
C_1^x \times \calA_2   & \text{if } (y_1,x_2) \notin C \\
\calA_1 \times C_2^x   & \text{if } (x_1,y_2) \notin C \\
C_1^y \times C_2^y     & \text{otherwise.} 
\end{cases}\]
\end{claimproof}

\begin{claim}
    There exists a clopen rectangle $R$ containing $x$ such that $R \subseteq C$.
\end{claim}
\begin{claimproof}
    Suppose not. Then $R \setminus C$ is nonempty for every clopen rectangle $R$ containing $x$. Since the intersection of finitely many clopen rectangles is again a clopen rectangle, it follows that 
    $$\mathcal{R}(x) = \{R \setminus C : R \text{ is a clopen rectangle containing } x\}$$
    is a collection of compact subsets of $\calB$ having the finite intersection property, which implies
    \begin{equation}
    \bigcap_{R \in \mathcal{R}(x)} R \setminus C \neq \emptyset,
\end{equation}
which in turn contradicts Claim~\ref{claim:randgaplabel:clopSepar}.
\end{claimproof}
\newline

The argument thus far shows that for every $x \in C$, there is a clopen rectangle $R =  R(x)$ with $x \in R\subseteq C$. This shows that $C$ may be written as a union of clopen rectangles. By compactness, $C$ may be written as a finite union of clopen rectangles. Finally, one can use induction and the identity
\begin{equation}
    (R_1 \times R_2) \cup (S_1 \times S_2)
    = [(R_1\cap S_1) \times (R_2 \cup S_2)] 
    \sqcup [(R_1\setminus S_1) \times R_2] 
    \sqcup [(S_1\setminus R_1) \times S_2]
\end{equation}
to show that any finite union of clopen rectangles may be represented as a finite disjoint union of clopen rectangles.
\end{proof}

Let us now give the 
\begin{proof}[Proof of Proposition~\ref{prop:cloptoQC}]
    Suppose $C \in \clopen(\Omega)$ is nonempty. For each $\omega \in C$, there is a basic open set $U(\omega) = \prod_n U_n(\omega)$ containing $\omega$ and contained in $C$, where each $U_n(\omega)$ is open and  $U_n(\omega) = \calA$ for all but finitely many $n$. By compactness, we can choose $\omega^{(1)},\ldots,\omega^{(m)} \in C$ so that
    \begin{equation}
        C = \bigcup_{j=1}^m U(\omega^{(j)}).
    \end{equation}
    
    Choose $N$ large enough that $U_n(\omega^{(j)}) = \calA$ for all $|n| > N$ and all $1 \le j \le m$, let $\pi_N:\Omega \to \calA^{\{-N,\ldots,N\}}$ denote the projection onto coordinates in $[-N,N]$, and note that
    \begin{equation}
        C_N:= \pi_N(C)
    \end{equation}
    is  clopen in $\calA^{\{-N,\ldots,N\}}$.
    \begin{claim*}
    We have
    \begin{equation}
        C = \pi_N^{-1}(C_N).
    \end{equation}
    \end{claim*}
    \begin{claimproof}
        Since $C_N = \pi_N(C)$, we immediately get 
        $$\pi_N^{-1}(C_N) = \pi_N^{-1}(\pi_N(C)) \supseteq C.$$
        For the other inclusion, assume $\omega \in \pi_N^{-1}(C_N)$. By definition, there is some $\omega' \in C$ such that $\omega_n = \omega_n'$ for all $|n| \leq N$. Choose $1 \le j \le m$ so that $\omega' \in U(\omega^{(j)})$. Since $\omega_n = \omega_n'$ for $|n| \leq N$ and $U_n(\omega^{(j)}) = \calA$ for $|n| > N$, it follows that $\omega \in U(\omega^{(j)}) \subseteq C$, as desired.
    \end{claimproof}

    Now, given $\omega \in C$, applying Lemma~\ref{lem:rand:GL:clopProdcloprect} (and induction), we get a clopen rectangle 
    $$
    R_{-N} \times \cdots \times R_N \subseteq C_N$$ containing $\pi_N(\omega)$. Extending $R_j = \calA$ for all $|j| > N$ gives a clopen rectangle in $\Omega$ containing $\omega$ and contained in $C$. As in the proof of Proposition~\ref{prop:cloptoQC}, this expresses $C$ as a union of clopen rectangles, which can be first reduced to a finite union via compactness and then decomposed into a disjoint union by algebraic considerations.
\end{proof}

\begin{prop} \label{prop:algebras}
$\bbZ[\mu[\QC(\Omega)]] = \bbZ[\mu_0[\clopen(\calA)]]$.
\end{prop}

\begin{proof}
    If $\Xi \in \QC(\Omega)$, we can write 
    $$
    \Xi = \prod_{j \in \bbZ} B_j,
    $$
    where each $B_j \subseteq \calA$ is clopen and all but finitely many are $\calA$. From this, one sees $\mu(\Xi) \in \bbZ[\mu_0[\clopen(\calA)]]$. Thus,
    \begin{equation}
        \bbZ[\mu[\QC(\Omega)]] \subseteq \bbZ[\mu_0[\clopen(\calA)]].
    \end{equation}
    On the other hand, for any $B \in \clopen(\calA)$, the set $\Xi  = \set{\omega \in \Omega : \omega_0 \in B }$ belongs to $\QC(\Omega)$ and satisfies $\mu(\Xi) = \mu_0(B)$. Therefore, $\mu_0[\clopen(\calA)] \subseteq \mu[\QC(\Omega)]$ and the corresponding inclusion holds also for the algebras they generate, which completes the proof.
\end{proof}

\begin{proof}[Proof of Theorem~\ref{t:rand:GL:main}]
    Since the characteristic function of any $\Xi \in \QC(\Omega)$ is a continuous integer-valued function, we have $\mu(\QC(\Omega)) \subseteq \mu(C(\Omega,\bbZ))$ and hence
    \[\bbZ[\mu(\QC(\Omega))] 
    \subseteq  \mu(C(\Omega,\bbZ)).\]
    On the other hand, any $f \in C(\Omega,\bbZ)$ can be written as
    \[ f = \sum_{m \in \mathrm{Ran}(f)} m \chi_{C_m},\]
    where $C_m = f^{-1}[\{m\}]$; note that $C_m$ is clopen for every $m$ and is empty for all but finitely many $m$. 
For each $m$ with $C_m \neq \emptyset$, we can write it as a disjoint union of clopen rectangles by Proposition~\ref{prop:cloptoQC}. 
This shows that $\int \! f \, d\mu$ belongs to $\bbZ[\mu(\QC(\Omega))]$, and thus
\begin{equation}\label{eq:ZmuCRmuCOmegaZ}
\bbZ[\mu(\QC(\Omega))] 
    =  \mu(C(\Omega,\bbZ)).
\end{equation}
Combining \eqref{eq:ZmuCRmuCOmegaZ} with Theorem~\ref{thm:rand:GL:main} and Proposition~\ref{prop:algebras}, one has
    \begin{equation}
    \frS(\Omega,T,\mu) 
    = \mu(C(\Omega,\bbZ)) 
    = \bbZ[\mu(\QC(\Omega))]
    = \bbZ[\mu_0[\clopen(\calA)]],
    \end{equation}
    as promised.
\end{proof}

\begin{remark} \label{rem:cylindertoclopen}
When $\calA$ is finite, a \emph{cylinder set} is a set of the form
\begin{equation}
\Xi_{u,I} = \set{\omega \in \calA^\bbZ : \omega|_I = u}
\end{equation}
where $u$ is a finite word, and $I = [a,b] \cap  \bbZ$ denotes an interval in $\bbZ$.
Then, every cylinder set is a clopen rectangle; indeed, writing $u = u_a\cdots u_b$, $R_n = \{u_n\}$ for $a \le n \le b$ and $R_n = \calA$ otherwise, one has
\[\Xi_{w,I} = \prod_{n \in \bbZ} R_n.\]
Moreover, every clopen rectangle is a union of disjoint cylinder sets. Concretely, if $R$ is a clopen rectangle, we may write
\[ R = \prod_{n \in \bbZ} R_n, \]
where every $R_n$ is clopen and choose $N$ for which $R_n = \calA$ for $|n| > N$. Denoting $P = \prod_{|n|\leq N}R_n$, we get
\[ R = \bigsqcup_{w \in P} \Xi_{w,[-N,N]}. \]
In particular,
\[ \{ \text{measures of clopen sets}\} = \{ \text{measures of finite disjoint unions of cylinder sets} \}. \]

\begin{theorem}\label{t:DTMclopen}
Suppose $(\Omega, T,\mu)$ is a full shift over a compact alphabet $\calA \subseteq \bbR$. 
Then
\begin{equation}
\diaglabels= \mu(\clopen(\Omega)) \cap (0,1).
\end{equation}
Equivalently, for $\ell \in \frS_0$,  $\OG_\diag(\ell) \not= \emptyset $ if and only if there is a clopen set $\Xi \subseteq \Omega$ with $\mu(\Xi) = \ell$.
\end{theorem}

We will need the following stability result.

\begin{lemma} \label{lem:preservationOfLabel}
Suppose $(\Omega,T,\mu)$ is an ergodic topological dynamical system with $\supp \mu = \Omega$ and $(\disamp,\odsamp) \in C(\Omega,\bbR) \times C(\Omega,\bbR_{\geq 0})$ is such that the family $\{J_{\disamp,\odsamp,\omega}\}$ has an open spectral gap $I$ with label $\ell$.
If
\begin{equation} \label{eq:gapStabilityBd}
\|\disamp - \disamp'\|_\infty + 2 \|\odsamp - \odsamp'\|_\infty
< |I|/2,
\end{equation}
then $\{J_{\disamp',\odsamp',\omega}\}$ also has an open spectral gap with label $\ell$.
\end{lemma}

The proof of the lemma uses a well-known perturbative fact about the spectrum. For compact $A,B \subseteq \bbR$, one defines the \emph{Hausdorff distance} between them by
\begin{align}
d_\Hausdorff(A,B)
& = \max\set{\sup_{a \in A}d(a,B), \sup_{b \in B} d(b,A)} \\
& = \inf\set{\varepsilon>0 : A \subseteq U_\varepsilon(B) \text{ and }B \subseteq U_\varepsilon(A)},
\end{align}
where $U_\varepsilon(S) = \{x \in \bbR : |x-y| \leq \varepsilon \text{ for some } y \in S\}$.
For any bounded self-adjoint operators $S,T$, one has 
\begin{equation} \label{eq:specdistST}
d_\Hausdorff(\sigma(S),\sigma(T)) \leq \|S - T \|.
\end{equation}
See \cite[Theorem~V.4.10]{Kato1980:PertTh} or \cite[Lemma~5.2.4]{ESO2} for a proof.

\begin{proof}[Proof of Lemma~\ref{lem:preservationOfLabel}]
Let $E_\star$ denote the midpoint of $I$, and choose $\omega$ in the full-measure set for which
\begin{align}
\label{eq:kEstarJ}
k_{\disamp,\odsamp}(E_\star)
& = \lim_{N \to \infty} \frac{1}{N} \tr \, \chi_{(-\infty,E_\star]}(J_{\disamp, \odsamp,\omega}\chi_{[0,N-1]}), \\
\label{eq:kEstarJ'}
k_{\disamp',\odsamp'}(E_\star)
& = \lim_{N \to \infty} \frac{1}{N} \tr \, \chi_{(-\infty,E_\star]}(J_{\disamp', \odsamp',\omega}\chi_{[0,N-1]}).
\end{align}
Assumption \eqref{eq:gapStabilityBd} together with \eqref{eq:specdistST} implies that there is an open interval $I'$ containing $E_\star$ in a spectral gap of $\{J_{\disamp',\odsamp',\omega}\}$ and moreover that 
\[ \tr \, \chi_{(-\infty,E_\star]}(J_{\disamp, \odsamp,\omega}\chi_{[0,N-1]})
=  \tr \, \chi_{(-\infty,E_\star]}(J_{\disamp', \odsamp',\omega}\chi_{[0,N-1]})\]
 for every $N \in \bbN$.
Together with \eqref{eq:kEstarJ} and \eqref{eq:kEstarJ'}, we see that 
\[k_{\disamp', \odsamp',\omega}(E) = k_{\disamp, \odsamp, \omega}(E)
= \ell\quad \text{ for all } E \in I',\]
as promised.
\end{proof}

\begin{proof}[Proof of Theorem~\ref{t:DTMclopen}]
Assume $\Xi \subseteq \Omega$ is clopen and has measure $0<\ell<1$ (i.e., $\Xi \neq \emptyset,\Omega$).
Take $\disamp = 1-\chi_{\Xi} = \chi_{\Omega \setminus \Xi}$, and consider the diagonal family $\{V_{\disamp,\omega}\}$. 
 Considering a finite cutoff $V_{\disamp,\omega,N}$, we see that $V_{\disamp,\omega,N}$ has eigenvalues $0$ and $1$, whose normalized multiplicities converge (almost surely) to $\ell$ and $1-\ell$, respectively.
\medskip

Conversely, assume $\ell \in \diaglabels$, and choose $\disamp \in C(\Omega,\bbR)$ such that $\{V_{\disamp,\omega}\}_{\omega \in \Omega}$ has an open spectral gap $(E_\ell^-,E_\ell^+)$ with label $\ell$.
Then, $\Xi = \{\omega \in \Omega: \disamp (\omega) \leq E_\ell^-\}$ is clopen and one has $k(E) = \mu(\Xi)$ for $E \in (E_\ell^-,E_\ell^+)$.
\end{proof}

\end{remark}

\section{Opening Gaps} \label{sec:opengaps}

Fix $m \in \bbN$, let $\calA_m = \{0,1,2,\ldots,m-1\}$, and let $\mu_0$ be a fully supported probability measure on $\calA_m$. Let us denote the full shift generated by $(\calA_m,\mu_0)$ by $(\Omega_m,T,\mu)$.
Given an $\omega \in \Omega_m$, we denote its $T$-orbit by
\begin{equation}
    \orb(\omega) = \set{T^n \omega : n \in \bbZ} \subseteq \Omega_m
\end{equation}
and say that $\omega$ is $T$-\emph{generic} if $\orb(\omega)$ is dense in $\Omega_m$. For a given $\disamp \in C(\Omega,\bbR)$ let $\Sigma_\disamp$ denote the $\mu$-almost sure spectrum of the family of ergodic Schr\"odinger operators $\{ H_{\disamp,\omega} \}$.
Using strong operator convergence, we have $\Sigma_\disamp = \sigma(H_{\disamp,\omega})$ for any $T$-generic $\omega \in \Omega_m$.

On account of \eqref{eq:specdistST}, the map sending $\disamp \in C(\Omega_m,\bbR)$ to $\Sigma_\disamp \subseteq \bbR$ is Lipschitz continuous with respect to the uniform metric on the domain and the Hausdorff metric on the codomain, that is,
\begin{equation} \label{eq:specdistfg}
d_{\Hausdorff}(\Sigma_{\disamp_1},\Sigma_{\disamp_2}) \leq \|\disamp_1 - \disamp_2\|_\infty, \quad \disamp_1,\disamp_2 \in C(\Omega_m,\bbR).
\end{equation}
Denote the discrete Laplacian on $\ell^2(\bbZ)$ by $\Delta$, that is, $[\Delta\psi](n) = \psi(n-1) + \psi(n+1)$.
Since $H_{\disamp,\omega} =\Delta+V_{\disamp,\omega}$, $\|\Delta\|=2$, and $\sigma(V_{\disamp,\omega}) = \range(\disamp)$ for $T$-generic $\omega$, one has $d_\Hausdorff(\Sigma_\disamp,\range(\disamp)) \leq 2$, which gives
\begin{equation} \label{eq:SigmafHDcons1}
\Sigma_\disamp \subseteq
U_2(\range(\disamp)) =  \bigcup_{\omega \in \Omega}[\disamp(\omega) - 2, \disamp(\omega)+2]
\end{equation}
and 
\begin{equation} \label{eq:SigmafHDcons2}
\Sigma_\disamp \cap[\disamp(\omega) - 2, \disamp(\omega)+2] \neq \emptyset \ \text{ for every } \omega.
\end{equation}

\begin{theorem} \label{t:openGapsForFullShift}
For Baire-generic $\disamp \in C(\Omega_m,\bbR)$, and every $N\in \bbN$, there exists $\lambda_0 = \lambda_0(\disamp,N)$ such that $\Sigma_{\lambda \disamp}$ has at least $N$ connected components for every $\lambda >\lambda_0$.

Indeed, the generic subset contains any $\disamp$ for which $\range(\disamp)$ has infinitely many connected components.
\end{theorem}

\begin{proof}
For each $N$, let  $Q_N \subseteq C(\Omega_m,\bbR)$ consist of those $\disamp$ for which $\range(\disamp)$ has at least $N$ connected components. Equivalently, $\disamp \in Q_N$ if and only if one can find pairwise disjoint closed intervals $I_1,\ldots, I_N$ such that 
\begin{equation} \label{eq:rangefInXiN}
\range(\disamp) \subseteq \bigcup_{j=1}^N I_j, \quad I_j \cap \range(\disamp) \neq \emptyset \text{ for every } j.
\end{equation}
Note that \eqref{eq:rangefInXiN} remains true if $\disamp$ is replaced by $\lambda \disamp$ and $I_j$ by $\lambda I_j$.

Combining \eqref{eq:rangefInXiN} with \eqref{eq:SigmafHDcons1} and \eqref{eq:SigmafHDcons2}, one obtains
\begin{equation}
\Sigma_{\lambda \disamp}\subseteq \bigcup_\omega [\lambda \disamp(\omega)-2,\lambda \disamp(\omega)+2]
\subseteq \bigcup_{j=1}^N U_2(\lambda I_j)
\end{equation}
and
\begin{equation}
\Sigma_{\lambda \disamp} \cap U_2(\lambda I_j) \neq \emptyset \text{ for each } j.
\end{equation}
Choosing $\lambda$ large enough, the $U_2(\lambda I_j)$ are pairwise disjoint, so $\Sigma_{\lambda \disamp}$ has at least $N$ connected components.

Since each $Q_N$ is open and dense, their intersection is a dense $G_\delta$ set with the desired property.
\end{proof}

\begin{remark}
In fact, the only feature of $\Omega_m$ that is used in the previous proof is its total disconnectedness, so the conclusion of Theorem~\ref{t:openGapsForFullShift} holds with $(\Omega_m,T,\mu)$ replaced by any topologically transitive dynamical system with totally disconnected phase space and fully-supported ergodic measure.
\end{remark}

Using the construction from the previous proof, we can establish  Theorem~\ref{thm:gapsOpenLargeCoupling}.
In the course of the proof, we will need the following notions: a function $\disamp:\Omega_m \to \bbR$ is \emph{locally constant} with \emph{window} $I \subseteq \bbZ$ if $I$ is finite and
\begin{equation}
\omega|_I = \omega'|_I \implies \disamp(\omega)= \disamp(\omega').
\end{equation}
Without loss of generality, one can always choose $I$ to be an interval.
Observe that any locally constant $\disamp$ is continuous on $\Omega_m$ and the set of locally constant functions is dense in $C(\Omega_m,\bbR)$, for instance, by the Stone--Weierstrass theorem.

\begin{proof}[Proof of Theorem~\ref{thm:gapsOpenLargeCoupling}]
To begin, apply Theorem~\ref{t:rand:GL:main} to see that the Schwartzman group associated with $(\Omega_m, T, \mu)$ is precisely 
$\bbZ[1/m] = \{j/m^r : \ r \in \bbZ_+, \ j \in \bbZ \},$
so in this scenario, 
\begin{equation}
\frS_0 = \{j/m^r : r \in \bbZ_+, 1 \le j \le m^r -1 \} .
\end{equation}

(a) For each $r \in \bbZ_+$, let $S(r) \subseteq C(\Omega_m, \bbR)$ denote the set of $\disamp$ for which all gaps of the form $j/m^r$ open for sufficiently large $\lambda$, that is, $\disamp \in S(r)$ if and only if there exists $\lambda_0 = \lambda_0(\disamp,r)$ such that $\lambda \disamp \in \OG_\Schrodinger(j/m^r)$ for all $\lambda \geq \lambda_0(\disamp,r)$ and every $1\le j \le m^r-1$.
By the Baire category theorem, it suffices to show that $S(r)$ contains a dense open subset of $C(\Omega_m,\bbR)$ for each $r\in \bbZ_+$.

Let $S'(r)$ denote the set of locally constant $\disamp \in C(\Omega_m,\bbR)$ with window of length $r$ taking $m^{r}$ distinct values.
In other words, for some $n \in \bbZ$ one has
\[ \disamp(\omega) =  \disamp(\omega') \iff \omega_{n+j}  = \omega'_{n+j} \ \text{ for all } 0 \le j \le r -1 . \]
For each $\disamp \in S'(r)$, let $\delta(\disamp) = \tfrac14 \min\{ |\disamp(\omega) - \disamp(\omega') | : \disamp(\omega) \neq \disamp(\omega') \}$.
We claim that
\begin{equation}
U(r):=
\bigcup_{R=r}^\infty \bigcup_{\disamp \in S'(R)} B(\disamp,\delta(\disamp))
\end{equation}
is open, dense, and contained in $S(r)$.
Indeed, $U(r)$ is a union of open sets, hence open.
To see that $U(r)$ is dense, notice that any $\disamp \in C(\Omega_m)$ can be perturbed to a locally constant function with window $I$ having length at least $r$,
 and then perturbed again to a locally constant function with window $I$ taking $m^{\#I}$ distinct values.

To see that $U(r) \subseteq S(r)$, let $\disamp \in U(r)$ be given and choose $R \geq r$ and $\widetilde \disamp \in S'(R)$ such that $\|\disamp - \widetilde \disamp\|_\infty < \delta:=\delta(\widetilde \disamp)$.
Without loss, assume $\widetilde \disamp$ has window $I = [0,R-1]\cap\bbZ$, write the distinct values of $\widetilde \disamp$ as $v_1< \cdots < v_{m^R}$, and note that
\[ 
\delta
= \frac{1}{4} \min\{ v_{i+1} - v_i : 1\le i \le m^R-1 \}. \]
For $\mu$-a.e.\ $\omega \in \Omega_m$, the operator $V_{\widetilde \disamp,\omega}$ has spectrum $\{v_j : 1 \le j \le m^R\}$ and hence has $m^R-1$ open spectral gaps.
It follows from the definitions that
\[ k(E) = \frac{j}{m^R}, \quad E \in (v_j,v_{j+1}), \quad 1 \le j \leq m^R-1, \]
where $k$ denotes the integrated density of states associated with the operator family $\{V_{\widetilde \disamp,\omega}\}_{\omega \in \Omega_m}$.

Now consider the Jacobi operator $A_{\disamp,\omega,\varepsilon}= \varepsilon \Delta + V_{\disamp,\omega}$.
Choosing $0<\varepsilon < \frac{1}{4} \delta$, we may deduce the following conclusions from eigenvalue perturbation theory of Hermitian matrices: the eigenvalues of $A_{\disamp,\omega,\varepsilon,N}$ all lie in the disjoint union
\[ K =\bigsqcup_{j=1}^{m^R} [v_j-2\varepsilon-\delta, v_j + 2\varepsilon + \delta]
\subseteq \bigsqcup_{j=1}^{m^R} [v_j-\tfrac32 \delta, v_j + \tfrac32 \delta],\]
and furthermore, the number of eigenvalues of $A_{\disamp,\omega,\varepsilon,N}$ in $[v_j - \tfrac32 \delta, v_j + \tfrac32  \delta]$ is the same as the multiplicity of the eigenvalue $v_j$ of the matrix $V_{\widetilde \disamp,\omega,N}$.

Synthesizing the observations in the previous paragraph, we see that $(v_j+ \tfrac32 \delta, v_{j+1}- \tfrac32 \delta)$ is (contained in) an open spectral gap of $A_{\disamp,\omega,\varepsilon}$ and that the IDS of the family $\{A_{ \disamp,\omega,\varepsilon}\}_{\omega \in \Omega_m}$ takes value $j/m^R$ in that spectral gap. 
Thus,  $H_{\disamp,\omega,\varepsilon} = \varepsilon^{-1} A_{\disamp,\omega,\varepsilon}= \Delta + \varepsilon^{-1}V_{\disamp,\omega}$ has an open spectral gap containing the rescaled interval $\varepsilon^{-1} \cdot( v_j+\tfrac32\delta ,  v_{j+1} - \tfrac32 \delta  )$ with the same label $j/m^R$.
In particular, $\lambda \disamp \in \OG_\Schrodinger(j/m^r)$ for all $1 \le j \le m^r-1$ and all $\lambda > 4\delta^{-1}$.
\bigskip

(b) This follows directly from (a).
\bigskip

(c)
Define $\disamp \in C(\Omega_m,\bbR)$ by
\begin{equation}
\disamp(\omega) = \sum_{n=1}^\infty \frac{2\omega_n}{(2m-1)^n}.
\end{equation}
The range of $\disamp$ is the Cantor set of real numbers in $[0,1]$ having at least one $(2m-1)$-ary expansion with no odd digits, and thus for $\mu$-a.e.\ $\omega \in \Omega_m$, $\sigma(V_{\disamp,\omega}) = C_{m}$  (recall that $V_{\disamp,\omega}$ denotes the diagonal operator given by multiplication with $\disamp(T^n\omega)$).
Equivalently, $C_m$ can be obtained by an iterated function procedure via
\begin{equation}
C_m = \bigcap_{n \geq 0} F_n,
\end{equation} 
where $F_0 = [0,1]$ and 
\[F_{n+1} = \bigcup_{j=0}^{m-1} f_j(F_n),
\quad f_j(x) = \frac{2j+x}{2m-1}. \]
By induction, $F_n$ has $m^n$ connected components, so we can define a function $g_n: \bbR \setminus F_n \to [0,1]$ by
\[ g_n(x) = \frac{\# \{ \text{connected components of } F_n \text{ that lie to the left of }x \}}{m^n}.\] 
For each $n \leq n'$, and $x \notin F_n$, $g_n(x) = g_{n'}(x)$, so we can define $g_\infty: \bbR \setminus C_m \to [0,1]$ by $g_\infty(x) = \lim g_n(x)$.
Since $g_\infty$ is uniformly continuous, we can uniquely extend to $g_\infty: \bbR \to [0,1]$.

By construction, the integrated density of states associated with the family $\{V_{\disamp,\omega}\}_{\omega \in \Omega_m}$ is $g_\infty$, so every gap label of the form $j/m^r$ with $1\leq j \leq m^r-1$ and $r \in \bbZ_+$ corresponds to an open gap. 
This shows that $\disamp \in \AGO_\diag$ and $(\disamp,0) \in \AGO_\jacobi$, so both sets are nonempty.
\end{proof}

By Lemma~\ref{lem:preservationOfLabel}, one can see that $\OG_\Schrodinger(\ell)$ is always \emph{open} in $C(\Omega_m, \bbR)$.
Often, one tries to furthermore show that $\OG_\Schrodinger(\ell)$ is dense for each $\ell$. However, this cannot be the case in the current setting: if $\OG_\Schrodinger(\ell)$ were dense for every $\ell \in \frS_0$, then
\[\AGO_\Schrodinger = \bigcap_{\ell \in \frS_0} \OG_\Schrodinger(\ell)\]
would be a dense $G_\delta$ subset of $C(\Omega_m,\bbR)$ by the Baire category theorem, contradicting Theorem~\ref{t:DTMempty}. 

In fact, one can show a stronger statement with a little more work: denseness fails uniformly across \emph{all} labels, that is,  $C(\Omega_m,\bbR) \setminus \OG_\Schrodinger(\ell)$ has nonempty interior for \emph{every} $\ell \in \frS_0$.

\begin{prop} \label{prop:fullShiftDTMnotdense}
Let $(\Omega_m,T,\mu)$ be as above. For any $\ell \in \frS_0$, $\OG_\Schrodinger(\ell)$ is not dense in $C(\Omega_m,\bbR)$. Indeed, the complement of $\OG_\Schrodinger(\ell)$ contains an $\ell$-dependent open ball centered at the zero function. \end{prop}

\begin{proof}
Let $\ell \in \frS_0$ be given, and choose $E \in (-2,2) = \operatorname{int}\sigma(\Delta)$ such that $k(E) = \ell$, where $k = k_0$ denotes the IDS associated to the free Laplacian. 
Consider $\disamp \in C(\Omega_m,\bbR)$ satisfying $\|\disamp\|_\infty < \varepsilon=\frac{1}{2}\mathrm{dist}(E,\bbR \setminus [-2,2])$. Observe that $a = \disamp(0^\bbZ) \in (-\varepsilon, \varepsilon)$, so we have 
\[ E \in (-2+a,2+a) \subseteq \Sigma_\disamp,\]
where the second inclusion is a consequence of \cite[Theorem 5.2.10]{ESO2}.
Moreover, by \cite[Lemma~3.1]{AvronSimon1983DMJ}
\[ k_\disamp(E-\|\disamp\|_\infty) \leq k_0(E) \leq k_\disamp(E + \|\disamp\|_\infty),
\]
so there exists some 
\[
E' \in [E-\|\disamp\|_\infty, E+ \|\disamp\|_\infty] 
\subseteq (E-\varepsilon, E+\varepsilon)
\subseteq (-2+a,2+a)\]
at which $k_\disamp(E')= k_0(E)= \ell$, and this suffices to demonstrate that the gap with label $\ell$ is closed.
\end{proof}

\section{Gaps That Cannot Open} \label{sec:nonopengaps}
At this point, we have addressed the set of labels as well as questions related to the ability to open suitable spectral gaps. 
In particular, we now know that for specific measures, every gap can in principle be opened by a suitable choice of operator family, or more concisely: $\OG_\Schrodinger(\ell) \neq \emptyset$ for every $\ell \in \frS_0$.
The purpose of this section is to investigate \emph{obstructions} to gap opening. One obstruction is purely dynamical: as stated before, having a dense set of periodic orbits (together with the standing assumption $\supp \mu = \Omega$) means that the spectrum must have dense interior and hence one cannot open all gaps with a single sampling function whenever the label set is dense in $[0,1]$.
The other obstruction is arithmetic: if the weights on a Bernoulli full shift are \emph{transcendental}, then certain gaps are arithmetically excluded for diagonal operators.

We begin by proving Theorem~\ref{thm:labelSetIncl}, which relates the sets of \emph{possible} gap labels for diagonal, Schr\"odinger, and Jacobi operator families.
\begin{proof}[Proof of Theorem~\ref{thm:labelSetIncl}]
Let $\ell \in \diaglabels$ be given. Given $\disamp \in \OG_\diag(\ell)$, the operator $V_\omega = V_{\disamp,\omega}$ has an open gap $\gamma = (E_-,E_+)$ with label $\ell$. 
For $0 < \varepsilon < |\gamma|/4$, the operator $\varepsilon \Delta + V_\omega$ has an open gap $\gamma' \supseteq (E_-+\varepsilon, E_+  - \varepsilon)$ with the same label, $\ell$ (by Lemma~\ref{lem:preservationOfLabel}). 
Consequently, $\Delta + \varepsilon^{-1}V_\omega = \varepsilon^{-1}(\varepsilon \Delta + V_\omega)$ also has a gap with the same label, so $\varepsilon^{-1}\disamp \in \OG_\Schrodinger(\ell)$. 
In particular, $\ell \in \schrlabels$.

Since every Schr\"odinger operator is also a Jacobi matrix, the second inclusion follows. 
The third inclusion is simply a restatement of the gap labelling theorem from \cite{DFZ}, so we are done.
\end{proof}

Next, we show that $\AGO_\Schrodinger$ is always empty in the setting of a full shift over a non-connected alphabet, that is, there is no single sampling function that can simultaneously open all spectral gaps. 

\begin{proof}[Proof of Theorem~\ref{t:DTMempty}]
This follows immediately from the following well-known relation, which can be shown with the help of strong operator approximation:
\begin{equation}
\Sigma_\disamp
= \overline{\bigcup_{\omega \text{ periodic}} \sigma(H_{\disamp,\omega})}.
\end{equation}
See \cite[Theorem 5.2.10]{ESO2} for a proof.
Thus, $\Sigma_\disamp$ has dense interior and hence is not a Cantor set. Since the label set associated with $(\Omega,T,\mu)$ is a dense subgroup of $\bbR$, it follows that $\Sigma_\disamp$ has infinitely many labels corresponding to closed gaps.
\end{proof}

In fact, the inability to open gaps appears to be somewhat more severe than that: there are choices of ergodic measures on the full shift for which certain labels seem to not correspond to open gaps at all.
The question is in general delicate: we here show a partial result that certain gaps cannot open for ergodic families of \emph{diagonal operators}.
In view of Theorem~\ref{thm:labelSetIncl}, this is \emph{not} dispositive of the inability to open gaps with a specific label in the \emph{Schr\"odinger} case.

Let us be more specific. As established in Theorem~\ref{thm:gapsOpenLargeCoupling} for certain choices of $T$-ergodic measures on the full shift, it is possible to find, given any possible gap label $\ell$, a sampling function $f$ such that the corresponding operator exhibits an open gap with that label. However, the full shift is not uniquely ergodic, and the selection of a $T$-ergodic measure is crucial.
Consider the setting where we have a two-letter alphabet $\calA=\{0,1\}$ and the measure $\mu_0(\{0\})=\beta$ and $\mu_0(\{1\})=1-\beta$ where $\beta$ is transcendental. 
We will presently show that there are possible labels $\ell \in \frS_0$ for which $\OG_\diag(\ell) = \emptyset$, that is, there is no $\disamp$ for which a gap with label $\ell$ opens for the operator family $\{V_{\disamp,\omega}\}$.

\begin{proof}[Proof of Theorem~\ref{thm:gapsClosed}]
Without loss of generality, assume $0<\beta<1/2$ (if not, simply interchange the roles of $0$ and $1$).
Due to Theorem~\ref{t:rand:GL:main}, 
\[\frS(\Omega,T,\mu) =
\bbZ[\beta,1-\beta]=
 \bbZ[\beta],
\] so,
in view of Theorem~\ref{t:DTMclopen}, it suffices to exhibit $\ell \in \bbZ[\beta]\cap (0,1) \setminus \mu[\clopen(\Omega)] $.

Let $Y$ be a clopen subset of $\Omega$. 
Applying Proposition~\ref{prop:cloptoQC}, we can write $Y$ as a disjoint union of clopen rectangles, each of which can in turn be written as a disjoint union of finitely many cylinder sets (see Remark~\ref{rem:cylindertoclopen}).
Thus, we can choose  $M \in \bbN$ sufficiently large and express $Y$ as the disjoint union of a finite collection $\mathcal{F}$ of sets of the form
\begin{equation} 
 S
 =\prod_{i=-\infty}^{-M} \mathcal{A} \times \prod_{i=-M+1}^{M-1} \{\alpha_i\} \times \prod_{i=M}^{\infty} \mathcal{A}, 
 \end{equation}
where each $\alpha_i$ is either $0$ or $1$.

Note that for each $S\in \mathcal{F}$, we have
\[
\mu(S)
=\prod_{i=-M+1}^{M-1}\mu_0(\alpha_i)=\beta^n(1-\beta)^{2M-1-n},
\] 
where $n$ denotes the number of $i$ for which $\alpha_i=0$. Therefore
\[\mu(Y)=\sum_{n=0}^{2M-1}c_n \beta^n (1-\beta)^{2M-1-n},\]
where $c_n$ is the number of $S\in \mathcal{F}$ that have precisely $n$ zeros as coordinates and thus
\begin{equation} \label{eq:cnconstraint}
c_n\in \bbZ \cap \left[0,\binom{2M-1}{n} \right].
\end{equation} 

Now let $\ell = \beta^2+ \beta$, which belongs to $\bbZ[\beta] \cap (0,1) = \frS_0$ by the assumption $0<\beta<1/2$. 
We will show that there does not exist a clopen set $Y$ such that $\mu(Y)=\beta+\beta^2.$

We define the following polynomials
\[
p_1(x) = \sum_{n=0}^{2M-1}c_nx^n(1-x)^{2M-1-n}, \quad p_2(x) = x+x^2.
\]
Assuming $\mu(Y)=\ell$, it follows that
\[
p_1(\beta) = p_2(\beta).
\]
However, since $\beta$ is transcendental, the equality of these polynomials would imply that $p_1(1)=p_2(1)$, which
implies $c_{2M-1}=2$. 
Since \eqref{eq:cnconstraint} forces $c_{2M-1}\in\{0,1\}$, this is a contradiction.
Therefore, the polynomials $p_1$ and $p_2$ are not equivalent, and equality cannot hold for transcendental $\beta$. 
Consequently,  $\OG_\diag(\ell)=\emptyset$.
\end{proof}

\begin{remark}
Any $\ell$ that cannot be expressed in the form $\sum_{i=0}^{n} c_i \beta^i (1-\beta)^{n-i}$, where $c_i$ is an integer between $0$ and $\binom{n}{i}$, satisfies $\OG_\diag(\ell) = \emptyset$. However, the straightforward method used to show that $\ell=\beta^2 + \beta$ cannot be represented in this form does not necessarily apply to other choices of $\ell$.
\end{remark}

\bibliographystyle{abbrv}

\bibliography{ref}

\end{document}